\documentclass{article}
\usepackage[english]{babel}
\usepackage{geometry,amsmath,amsthm,amssymb,latexsym}
\usepackage{authblk}

\usepackage[colorlinks=true,linkcolor=blue,citecolor=blue,pdfpagelabels=false]{hyperref}

\usepackage{thmtools}
\theoremstyle{plain}
  \declaretheorem[numberwithin=section]{theorem}
  \declaretheorem[numberlike=theorem]{corollary}
  \declaretheorem[numberlike=theorem]{proposition}
  \declaretheorem[numberlike=theorem]{lemma}
  
  \declaretheorem[numberlike=theorem]{question}
\theoremstyle{definition}
  \declaretheorem[numberlike=theorem]{definition}
  \declaretheorem[numberlike=theorem]{example}
  \declaretheorem[numberlike=theorem]{remark}
\newenvironment{acknowledgements}{\bigskip\textbf{Acknowledgements.}}{}

\newcommand{\md}{\mathrm{d}}

\begin{document}

\title{Gauss congruences for rational functions in several variables}

\author[1]{Frits Beukers}
\author[1]{Marc Houben}
\author[2]{Armin Straub}
\affil[1]{Utrecht University}
\affil[2]{University of South Alabama}

\date{October 1, 2017}

\maketitle

\begin{abstract}
  We investigate necessary as well as sufficient conditions under which the
  Laurent series coefficients $f_{\boldsymbol{n}}$ associated to a multivariate
  rational function satisfy Gauss congruences, that is $f_{\boldsymbol{m}p^r}
  \equiv f_{\boldsymbol{m}p^{r - 1}}$ modulo $p^r$. For instance, we show that
  these congruences hold for certain determinants of logarithmic derivatives.
  As an application, we completely classify rational functions $P / Q$
  satisfying the Gauss congruences in the case that $Q$ is linear in each
  variable.
\end{abstract}

\section{Introduction}

We say that a sequence $(a_k)_{k \geq 0}$ of rational numbers satisfies
the {\emph{Gauss congruences}} for the prime $p$, if $a_k \in \mathbb{Z}_p$
(that is, the $a_k$ are $p$-adically integral) and
\begin{equation}
  a_{m p^r} \equiv a_{m p^{r - 1}} \quad (\operatorname{mod} p^r) \label{eq:gauss:1}
\end{equation}
for all integers $m \geq 0$ and $r \geq 1$. These congruences hold
for all primes if and only if
\begin{equation*}
  \sum_{d|m} \mu \left(\frac{m}{d} \right) a_d \equiv 0 \quad (\operatorname{mod}
   m),
\end{equation*}
where $\mu$ is the M{\"o}bius function, and they are named after the classical
congruences that hold in the case $a_k = \alpha^k$, with $\alpha \in
\mathbb{Z}$. We refer to \cite{zarelua-tr} and \cite{minton-cong} for a
survey of these and related congruences. Well-known examples of sequences
satisfying the Gauss congruences for all primes include the Lucas numbers
$L_n$ defined by $L_{n + 1} = L_n + L_{n - 1}$, with $L_0 = 2$, $L_1 = 1$, and
the Ap\'ery numbers
\begin{equation}
  A_n = \sum_{k = 0}^n \binom{n}{k}^2 \binom{n + k}{k}^2, \label{eq:apery}
\end{equation}
which featured in Ap\'ery's proof \cite{apery} of the irrationality of
$\zeta (3)$. In fact, as shown in \cite{beukers-apery85},
\cite{coster-sc}, the Ap\'ery numbers have the remarkable (and rare)
property of satisfying \eqref{eq:gauss:1} modulo $p^{3 r}$ if $p \geq 5$
(often referred to as a supercongruence).

In this paper, we consider the case of multivariate sequences
$(a_{\boldsymbol{k}})_{\boldsymbol{k} \in \mathbb{Z}^n}$. As in the univariate
case, these are said to satisfy the {\emph{Gauss congruences}} for the prime
$p$, if $a_{\boldsymbol{k}} \in \mathbb{Z}_p$ and $a_{\boldsymbol{m}p^r} \equiv
a_{\boldsymbol{m}p^{r - 1}} \pmod{p^r}$ for all $\boldsymbol{m}
\in \mathbb{Z}^n$ and all $r \geq 1$. Our particular focus is on the
case when the $a_{\boldsymbol{k}}$ are the coefficients of a Laurent series of a
rational function. As reviewed in Section~\ref{sec:laurent}, a rational
function $f = P / Q$ has Laurent series associated with each vertex of the
Newton polytope $N (Q)$ of $Q$. We show that Gauss congruences hold for one of
these Laurent series (for all but finitely many primes) if and only if they
hold for all Laurent series (Proposition~\ref{prop:PQv}), in which case we say
that $f$ has the {\emph{Gauss property}}.

As observed in \cite{s-apery}, the rational function
\begin{equation}
  \frac{1}{(1 - x_1 - x_2) (1 - x_3 - x_4) - x_1 x_2 x_3 x_4} =
  \sum_{\boldsymbol{k} \in \mathbb{Z}_{\geq 0}^4} A_{\boldsymbol{k}}
  \boldsymbol{x}^{\boldsymbol{k}}, \label{eq:apery:rat}
\end{equation}
where $\boldsymbol{x}^{\boldsymbol{k}}$ is short for $x_1^{k_1} x_2^{k_2}
x_3^{k_3} x_4^{k_4}$, has the Ap\'ery numbers \eqref{eq:apery} as its
diagonal coefficients, that is, $A_{n, n, n, n} = A_n$. Moreover, it is proved
in \cite{s-apery} that the supercongruences for the Ap\'ery numbers hold
for all coefficients $A_{\boldsymbol{n}}$, meaning that $A_{\boldsymbol{m}p^r}
\equiv A_{\boldsymbol{m}p^{r - 1}} \pmod{p^{3 r}}$ for all
primes $p \geq 5$. In particular, the rational function
\eqref{eq:apery:rat} has the Gauss property.

One of the goals of this paper is to address the question of which rational
functions have the Gauss property. Towards that end, we provide several
results that show that the Gauss property holds for large natural classes of
rational functions in several variables $\boldsymbol{x}= (x_1, x_2, \ldots,
x_n)$. For instance, in Section~\ref{sec:det}, we show that certain
determinants of logarithmic derivatives have the Gauss property. The following
is derived (as Theorem~\ref{thm:det:rat}) from a similar result for Laurent
series (Theorem~\ref{thm:det}). Because of its central character (indicated in
Question~\ref{q:det} below) and for future applications, the results of
Section~\ref{sec:det} are proved over more general rings, namely domains with
a Frobenius lift (Definition~\ref{def:frob}).

\begin{theorem}
  \label{thm:det:intro}Let $m \leq n$ and let $f_1, \ldots, f_m \in
  \mathbb{Q} (\boldsymbol{x})$ be nonzero. Then the rational function
  \begin{equation}
    \frac{x_1 \cdots x_m}{f_1 \cdots f_m} \det \left(\frac{\partial
    f_j}{\partial x_i} \right)_{i, j = 1, \ldots, m} \label{eq:det:intro}
  \end{equation}
  has the Gauss property.
\end{theorem}

It would be of considerable interest to fully characterize multivariate
rational functions with the Gauss property. Towards that end, one might be
tempted to ask the following question.

\begin{question}
  \label{q:det}Suppose that the rational function $f \in \mathbb{Q}
  (\boldsymbol{x})$ has the Gauss property. Can it be written as a
  $\mathbb{Q}$-linear combination of functions of the form
  \eqref{eq:det:intro}?
\end{question}

A recent result of Minton \cite{minton-cong} answers Question~\ref{q:det}
affirmatively when $n = 1$, the case of a single variable. For the benefit of
the reader and in order to be self-contained, we reprove this result, cast in
our present language, in Section~\ref{sec:minton}.

\begin{theorem}[Minton, 2014]
  A rational function $f \in \mathbb{Q} (x)$ has the Gauss property if and
  only if $f$ is a $\mathbb{Q}$-linear combination of functions of the form
  $x u' (x) / u (x)$, with $u \in \mathbb{Z} [x]$.
\end{theorem}

Although characterizing multivariate rational functions with the Gauss
property remains an open challenge, we obtain, in Section~\ref{sec:linear},
the following concise classification in the case of rational functions $P /
Q$, for which the Newton polytope $N (Q)$ of the denominator is contained in
$\{ 0, 1 \}^n$. Note that, in that case, the vertices of $N (Q)$ equal the
support of $Q$.

\begin{theorem}
  \label{thm:linear:intro}Let $P, Q \in \mathbb{Z} [\boldsymbol{x}]$ and
  suppose that $Q$ is linear in each variable. Then $P / Q$ has the Gauss
  property if and only if $N (P) \subseteq N (Q)$.
\end{theorem}

As illustrated by \eqref{eq:apery:rat}, such results for (multivariate)
rational functions allow us to establish congruences for numbers, such as the
Ap\'ery numbers, whose generating function is much more complicated than a
rational function. Indeed, observe that Theorem~\ref{thm:linear:intro}
immediately implies that the rational function \eqref{eq:apery:rat} has the
Gauss property. In particular, it follows that the Ap\'ery numbers satisfy
the Gauss congruences. In a similar spirit, recent results of Rowland and
Yassawi \cite{ry-diag13} show that the series coefficients of certain
rational functions satisfy Lucas congruences. Their approach using Cartier
operators can also be applied to provide an alternative proof (at least in
parts) of the ``if'' portion of Theorem~\ref{thm:linear:intro} (due to the
technicalities involved, we do not pursue this path here).

We obtain Theorem~\ref{thm:linear:intro} as an immediate consequence of the
following more general result, which we prove in Section~\ref{sec:linear} as
an application of Theorem~\ref{thm:det:intro}.

\begin{theorem}
  \label{thm:mostlylinear:intro}Let $P, Q \in \mathbb{Z} [z, \boldsymbol{x}]$
  such that $Q$ is linear in the variables $x_1, \ldots, x_n$. Write $P =
  \sum_{\boldsymbol{k}} p_{\boldsymbol{k}} (z) \boldsymbol{x}^{\boldsymbol{k}}$ and $Q
  = \sum_{\boldsymbol{k}} q_{\boldsymbol{k}} (z) \boldsymbol{x}^{\boldsymbol{k}}$ with
  $p_{\boldsymbol{k}}, q_{\boldsymbol{k}} \in \mathbb{Z} [z]$. Then $P / Q$ has
  the Gauss property if and only if $p_{\boldsymbol{k}} \neq 0$ implies
  $q_{\boldsymbol{k}} \neq 0$ and $p_{\boldsymbol{k}} / q_{\boldsymbol{k}}$ has the
  Gauss property for all $\boldsymbol{k}$ with $q_{\boldsymbol{k}} \neq 0$.
\end{theorem}

Our proof of Theorem~\ref{thm:mostlylinear:intro} answers Question~\ref{q:det}
affirmatively for rational functions $f = P / Q$, in the case that $Q$ is
linear in all but one variable. We further show in Example~\ref{eg:q:degree2},
that the answer to Question~\ref{q:det} is affirmative in the case that $Q$ is
a function of two variables and total degree $2$.

Although, in general, Question~\ref{q:det} remains far from being answered, we
can give a number of necessary conditions for the Gauss property to hold. A
simple such condition, proved in Proposition~\ref{prop:NPNQ}, is that the
Newton polytope $N (P)$ of $P$ must be contained in $N (Q)$. As another
example, made precise in Proposition~\ref{prop:PQF}, consider a face $F$ of $N
(Q)$ and let $P_F, Q_F$ be the restrictions of $P, Q$ consisting of those
monomials supported on $F$. If $P / Q$ has the Gauss property, then the same
holds for $P_F / Q_F$. In Proposition~\ref{prop:toroidal}, we prove the
straightforward observation that toroidal substitutions preserve the Gauss
property. As a consequence, illustrated in Example~\ref{eg:toroidal:rat}, the
rational function $P_F / Q_F$ can be reduced to a rational function in
essentially fewer than $n$ variables.

Finally, let us indicate a useful consequence concerning arbitrary
substitutions of an affirmative answer to Question~\ref{q:det}. Suppose that
$f \in \mathbb{Q} (\boldsymbol{x})$ is a $\mathbb{Q}$-linear combination of
functions of the form \eqref{eq:det:intro}, and let $g_1, \ldots, g_n \in
\mathbb{Q} (\boldsymbol{x})$ be nonzero. Then, by the multivariate chain rule,
\begin{equation}
  \frac{x_1 \cdots x_n}{g_1 \cdots g_n} \det \left(\frac{\partial
  g_j}{\partial x_i} \right)_{i, j = 1, \ldots, n} f (g_1 (\boldsymbol{x}),
  \ldots, g_n (\boldsymbol{x})) \label{eq:msubst:rat:intro}
\end{equation}
also is a $\mathbb{Q}$-linear combination of functions of the form
\eqref{eq:det:intro}. In particular, by Theorem~\ref{thm:det:intro}, the
rational function \eqref{eq:msubst:rat:intro} has the Gauss property. Hence,
if Question~\ref{q:det} has an affirmative answer, then it follows that, for
any rational function $f \in \mathbb{Q} (\boldsymbol{x})$ with the Gauss
property, the rational function \eqref{eq:msubst:rat:intro} has the Gauss
property as well.

Since Question~\ref{q:det} remains open, we give a direct and independent
proof of the following univariate version in Section~\ref{sec:subst}.

\begin{theorem}
  \label{thm:subst:rat:intro}Let $g_j \in \mathbb{Q} (x)$ be nonzero. If the
  rational function $f \in \mathbb{Q} (\boldsymbol{x})$ has the Gauss property,
  then so does the rational function
  \begin{equation}
    \left(\prod_{j = 1}^n \frac{x_j g'_j (x_j)}{g_j (x_j)} \right) f (g_1
    (x_1), \ldots, g_n (x_n)) . \label{eq:subst:rat:intro}
  \end{equation}
\end{theorem}

\section{Preliminaries and Laurent series expansions}\label{sec:laurent}

Throughout, $p$ is a prime. The {\emph{$p$-adic valuation}} $\nu_p (a)$ of a
rational number $a \in \mathbb{Q}^{\times}$ is the largest $r \in
\mathbb{Z}$ such that $a / p^r \in \mathbb{Z}_p$, with the understanding
that $\nu_p (0) = \infty$. If $\boldsymbol{a}= (a_1, \ldots, a_n) \in
\mathbb{Q}^n$ is a vector, then $\nu_p (\boldsymbol{a}) = \min \{ \nu_p (a_1),
\ldots, \nu_p (a_n) \}$. Similarly, we say that $p$ divides a vector
$\boldsymbol{a}$ if and only if $\nu_p (\boldsymbol{a}) \geq 1$, that is, $p$
divides each component of $\boldsymbol{a}$.

When working with several variables, we typically use the vector notation
$\boldsymbol{x}= (x_1, \ldots, x_n)$ and write, for instance, $\mathbb{Q}
(\boldsymbol{x}) =\mathbb{Q} (x_1, \ldots, x_n)$ for the ring of rational
functions, or $\mathbb{Q} [\boldsymbol{x}^{\pm 1}] =\mathbb{Q} [x_1^{\pm 1},
\ldots, x_n^{\pm 1}]$ for the ring of Laurent polynomials in several
variables. Similarly, we write $\boldsymbol{x}^{\boldsymbol{k}} = x_1^{k_1} \cdots
x_n^{k_n}$ for monomials and refer to $\boldsymbol{k}= (k_1, \ldots, k_n)$ as
its exponent vector. The {\emph{support}} of a Laurent polynomial $P \in
\mathbb{Q} [\boldsymbol{x}^{\pm 1}]$, denoted $\operatorname{supp} (P) \subseteq
\mathbb{Z}^n$, is the set of exponent vectors of the (non-zero) monomials of
$P$. The {\emph{Newton polytope}} $N (P)$ of $P$ is the convex closure of
$\operatorname{supp} (P)$. The {\emph{cone}} generated by vectors $\boldsymbol{v}_1,
\boldsymbol{v}_2, \ldots \in \mathbb{R}^n$ is the $\mathbb{R}_{\geq
0}$-span of these vectors. We say such a cone $C$ is {\emph{proper}} if there
exists a linear form $\alpha$ such that $\alpha (\boldsymbol{w}) > 0$ for all
nonzero $\boldsymbol{w} \in C$. The cones of importance to us are proper. For
instance, suppose that $\boldsymbol{v}$ is a vertex of the Newton polytope $N
(P)$ of a Laurent polynomial $P$. Then the cone generated by $N
(P\boldsymbol{x}^{-\boldsymbol{v}})$ is a proper cone. Note that our cones are
based on the vertex $\boldsymbol{0}$.

In this paper, we frequently discuss rational functions $F = P / Q \in
\mathbb{Q} (\boldsymbol{x})$. In principle, we could choose $P$ and $Q$ to be
polynomials. However, for certain purposes, it turns out to be natural to
allow $P$ and $Q$ to be Laurent polynomials, that is, $P, Q \in \mathbb{Z}
[\boldsymbol{x}^{\pm 1}]$. We have a Laurent series expansion of $F$ associated
to each vertex $\boldsymbol{v}$ of $N (Q)$ as follows \cite[p.~195]{gkz-det}.
Writing $Q = \sum_{\boldsymbol{k}} q_{\boldsymbol{k}} \boldsymbol{x}^{\boldsymbol{k}}$
with $q_{\boldsymbol{k}} \in \mathbb{Z}$, note that
$\boldsymbol{x}^{\boldsymbol{v}} / Q$ can be expanded as
\begin{equation}
  \frac{\boldsymbol{x}^{\boldsymbol{v}}}{Q} = \frac{1}{q_{\boldsymbol{v}} +
  \sum_{\boldsymbol{k} \neq \boldsymbol{v}} q_{\boldsymbol{k}}
  \boldsymbol{x}^{\boldsymbol{k}-\boldsymbol{v}}} = \frac{1}{q_{\boldsymbol{v}}}
  \sum_{m = 0}^{\infty} \left(- \sum_{\boldsymbol{k} \neq \boldsymbol{v}}
  \frac{q_{\boldsymbol{k}}}{q_{\boldsymbol{v}}}
  \boldsymbol{x}^{\boldsymbol{k}-\boldsymbol{v}} \right)^m = \sum_{\boldsymbol{k} \in
  C} g_{\boldsymbol{k}} \boldsymbol{x}^{\boldsymbol{k}}, \label{eq:laurent:Q}
\end{equation}
where $C$ is the proper cone generated by the vectors $N (Q
/\boldsymbol{x}^{\boldsymbol{v}})$. To see that the $g_{\boldsymbol{k}}$ in
\eqref{eq:laurent:Q} are finite, so that the series is well-defined, let
$\alpha$ be a linear form such that $\alpha (\boldsymbol{w}) > 0$ for all
nonzero $\boldsymbol{w} \in C$, and observe that $\alpha
(\boldsymbol{k}-\boldsymbol{v}) > 0$ for all $\boldsymbol{k} \in \operatorname{supp} (Q)$
with $\boldsymbol{k} \neq \boldsymbol{v}$. Clearly, $g_{\boldsymbol{k}} \in
\mathbb{Z} [q_{\boldsymbol{v}}^{- 1}]$. Multiplying the series
\eqref{eq:laurent:Q} with $P /\boldsymbol{x}^{\boldsymbol{v}}$, we obtain a
Laurent series $f = \sum_{\boldsymbol{k}} f_{\boldsymbol{k}}
\boldsymbol{x}^{\boldsymbol{k}}$ for $F = P / Q$. Its coefficients again satisfy
$f_{\boldsymbol{k}} \in \mathbb{Z} [q_{\boldsymbol{v}}^{- 1}]$. We refer to $f
\in \mathbb{Q} [[\boldsymbol{x}^{\pm 1}]]$ as the {\emph{Laurent series
expansion}} of $F$ with respect to $\boldsymbol{v}$. The support of $f$, denoted
by $\operatorname{supp} (f)$, is the set of all $\boldsymbol{k} \in \mathbb{Z}^n$ such
that $f_{\boldsymbol{k}} \neq 0$. Observe that, if $p$ is a prime such that
$q_{\boldsymbol{v}} \in \mathbb{Z}_p^{\times}$, then $f \in \mathbb{Z}_p
[[\boldsymbol{x}^{\pm 1}]]$. In particular, for all but finitely many primes
$p$, all Laurent series expansions of $F$ have coefficients that are $p$-adic
integers.

For any Laurent series $f = \sum_{\boldsymbol{k}} f_{\boldsymbol{k}}
\boldsymbol{x}^{\boldsymbol{k}} \in \mathbb{Q} [[\boldsymbol{x}^{\pm 1}]]$, we
refer to $f_{\boldsymbol{0}}$ as its constant term.

Finally, especially in Section~\ref{sec:det}, it will be convenient to use the
Euler operator $\theta_y = y \frac{\partial}{\partial y}$. If there is no
possibility for confusion, we simply write $\theta_i = \theta_{x_i}$.

\section{Gauss congruences}

\begin{definition}
  \label{def:gauss:laurent}We say that a Laurent series
  \begin{equation*}
    f = \sum_{\boldsymbol{k} \in \mathbb{Z}^n} f_{\boldsymbol{k}}
     \boldsymbol{x}^{\boldsymbol{k}} \in \mathbb{Q} [[\boldsymbol{x}^{\pm 1}]]
  \end{equation*}
  satisfies the {\emph{Gauss congruences}} for the prime $p$ if $f \in
  \mathbb{Z}_p [[\boldsymbol{x}^{\pm 1}]]$ and
  \begin{equation*}
    f_{\boldsymbol{m}p^r} \equiv f_{\boldsymbol{m}p^{r - 1}} \quad (\operatorname{mod}
     p^r)
  \end{equation*}
  for all $\boldsymbol{m} \in \mathbb{Z}^n$ and all $r \geq 1$. We say
  that $f$ has the {\emph{Gauss property}} if it satisfies the Gauss
  congruences for all but finitely many primes.
\end{definition}

Let $U_p$ be the operator on Laurent series defined by
\begin{equation}
  U_p \left(\sum_{\boldsymbol{k} \in \mathbb{Z}^n} c_{\boldsymbol{k}}
  \boldsymbol{x}^{\boldsymbol{k}} \right) = \sum_{\boldsymbol{k} \in \mathbb{Z}^n}
  c_{p\boldsymbol{k}} \boldsymbol{x}^{\boldsymbol{k}} . \label{eq:U}
\end{equation}
Note that a Laurent series $f$ has the Gauss property if and only if, for all
$r \geq 1$ and all but finitely many primes $p$,
\begin{equation*}
  U_p^r (f) \equiv U_p^{r - 1} (f) \quad (\operatorname{mod} p^r) .
\end{equation*}
The following observation is more or less straightforward.

\begin{proposition}
  \label{prop:Up}Let $\zeta = e^{2 \pi i / p}$. If $f$ is the Laurent series
  of the rational function $F \in \mathbb{Q} (\boldsymbol{x})$ with respect to
  the vertex $\boldsymbol{v}$, then $U_p (f)$ is the Laurent series of the
  rational function
  \begin{equation*}
    F^{(p)} = \frac{1}{p^n} \sum_{r_1, \ldots, r_n = 0}^{p - 1} F
     (\zeta^{r_1} x_1^{1 / p}, \ldots, \zeta^{r_n} x_n^{1 / p})
  \end{equation*}
  with respect to the vertex $p^{n - 1} \boldsymbol{v}$.
\end{proposition}

In fact, let us note that, with $F = P / Q$ for $P, Q \in \mathbb{Z}
[\boldsymbol{x}^{\pm 1}]$, we can write $F^{(p)} = P^{(p)} / Q^{(p)}$ with
denominator
\begin{equation*}
  Q^{(p)} = \prod_{r_1, \ldots, r_n = 0}^{p - 1} Q (\zeta^{r_1} x_1^{1 / p},
   \ldots, \zeta^{r_n} x_n^{1 / p}),
\end{equation*}
which is in $\mathbb{Z} [\boldsymbol{x}^{\pm 1}]$ with Newton polytope $N
(Q^{(p)}) = p^{n - 1} N (Q)$. If $\boldsymbol{v}$ is a vertex of $N (Q)$ with
coefficient $q_{\boldsymbol{v}} \in \mathbb{Z}_p^{\times}$, then $p^{n - 1}
\boldsymbol{v}$ is a vertex of $N (Q^{(p)})$ with coefficient
$q_{\boldsymbol{v}}^{p^n} \in \mathbb{Z}_p^{\times}$.

\begin{definition}
  A rational function $F = P / Q \in \mathbb{Q} (\boldsymbol{x})$ has the
  {\emph{Gauss property}} if, for every vertex $\boldsymbol{v}$ of $N (Q)$, the
  Laurent series expansion of $F$ with respect to $\boldsymbol{v}$ has the Gauss
  property.
\end{definition}

By the next result, it suffices to consider a single vertex $\boldsymbol{v}$ in
that definition. The proof relies on the following simple but important
observation. If $P, Q \in \mathbb{Z} [\boldsymbol{x}^{\pm 1}]$ and
$\boldsymbol{v}$ is a vertex of $N (Q)$, then, for all but finitely many primes
$p$, the Laurent expansion $f$ of $P / Q$ with respect to $\boldsymbol{v}$ has
$p$-adic integers as coefficients. In such a case, $f \equiv 0 \pmod{p^r}$ if and only if $P \equiv 0 \pmod{p^r}$.

\begin{proposition}
  \label{prop:PQv}Let $P, Q \in \mathbb{Z} [\boldsymbol{x}^{\pm 1}]$. Let
  $\boldsymbol{v}, \boldsymbol{w}$ be vertices of $N (Q)$. Then the Laurent series
  expansion of $F = P / Q$ with respect to $\boldsymbol{v}$ has the Gauss
  property if and only if the Laurent series expansion of $F$ with respect to
  $\boldsymbol{w}$ has the Gauss property.
\end{proposition}

\begin{proof}
  Suppose that $f = \sum_{\boldsymbol{k} \in \mathbb{Z}^n} f_{\boldsymbol{k}}
  \boldsymbol{x}^{\boldsymbol{k}}$ is the Laurent series expansion of $F$ with
  respect to $\boldsymbol{v}$. By Proposition~\ref{prop:Up}, $U_p^r (f) \in
  \mathbb{Z}_p [[\boldsymbol{x}^{\pm 1}]]$ is the Laurent expansion of a
  rational function $P^{(p^r)} / Q^{(p^r)}$ with respect to $p^{r (n - 1)}
  \boldsymbol{v}$. Consequently, $U_p^r (f) - U_p^{r - 1} (f)$ is the Laurent
  expansion of a rational function $R_{p, r}$ with denominator $Q^{(p^r)}
  Q^{(p^{r - 1})} \in \mathbb{Z} [\boldsymbol{x}^{\pm 1}]$. This expansion is
  with respect to the vertex $p^{(2 r - 1) (n - 1)} \boldsymbol{v}$. Note that
  the rational function $R_{p, r}$ is independent of the choice of
  $\boldsymbol{v}$.
  
  Let $p$ be a prime such that $f \in \mathbb{Z}_p [[\boldsymbol{x}^{\pm 1}]]$.
  Then, the corresponding Laurent expansion $U_p^r (f) - U_p^{r - 1} (f)$ of
  $R_{p, r}$ has $p$-adic integers as coefficients as well. Consequently, the
  congruence $U_p^r (f) - U_p^{r - 1} (f) \equiv 0 \pmod{p^r}$ holds if and only if the numerator of $R_{p, r}$ is divisible by
  $p^r$.
  
  Hence, $f$ has the Gauss property if and only if, for all but finitely many
  primes $p$, the numerator of $R_{p, r}$ is divisible, for all $r \geq
  1$, by $p^r$. Since this latter statement is independent of the choice of
  vertex $\boldsymbol{v}$, the claim follows.
\end{proof}

The proof of the next observation actually only makes use of the fact that $P
/ Q$ satisfies the Gauss congruences for a single suitable prime $p$.

\begin{proposition}
  \label{prop:NPNQ}Let $P, Q \in \mathbb{Z} [\boldsymbol{x}^{\pm 1}]$. If $P /
  Q$ has the Gauss property, then $N (P) \subseteq N (Q)$.
\end{proposition}

\begin{proof}
  To prove this claim it is sufficient to show that, for every vertex
  $\boldsymbol{v} \in N (Q)$, the Newton polytope $N (P
  /\boldsymbol{x}^{\boldsymbol{v}})$ is contained in the cone $C$ generated by the
  vectors $N (Q /\boldsymbol{x}^{\boldsymbol{v}})$. Let $p$ be a prime such that
  the Gauss congruences for $P / Q$ hold for $p$ and such that the
  coefficients of $P$ corresponding to vertices of $N (P)$ are $p$-adic units.
  
  Let $\boldsymbol{v} \in N (Q)$, and let $\sum_{\boldsymbol{k}} f_{\boldsymbol{k}}
  \boldsymbol{x}^{\boldsymbol{k}}$ be the Laurent series for $P / Q$ with respect
  to $\boldsymbol{v}$. For contradiction, suppose that there is a vertex
  $\boldsymbol{w}$ of $N (P /\boldsymbol{x}^{\boldsymbol{v}})$ such that
  $\boldsymbol{w} \not\in C$. It follows from \eqref{eq:laurent:Q} that
  $\boldsymbol{w}$ is a vertex of the convex hull of the support of the Laurent
  series of $P / Q$ with respect to $\boldsymbol{v}$. By our choice of $p$,
  $f_{\boldsymbol{w}}$ is a $p$-adic unit. On the other hand, the point
  $p\boldsymbol{w}$ is not in the support of the Laurent series, so that
  $f_{p\boldsymbol{w}} = 0$. This contradicts the congruence $f_{p\boldsymbol{w}}
  \equiv f_{\boldsymbol{w}} \pmod{p}$.
\end{proof}

\begin{example}
  The rational function
  \begin{equation*}
    \frac{P}{Q} = \frac{1 + 2 x - x^2}{1 - x^2} = \frac{1}{1 - x} +
     \frac{x}{1 + x}
  \end{equation*}
  obviously satisfies the Gauss congruences for all primes. As predicted by
  Proposition~\ref{prop:NPNQ}, $N (P) \subseteq N (Q)$. However, note that
  $\operatorname{supp} (P) \nsubseteq \operatorname{supp} (Q)$.
\end{example}

\begin{corollary}
  \label{cor:NQ:cone}Suppose $F = P / Q$ with $P, Q \in \mathbb{Z}
  [\boldsymbol{x}^{\pm 1}]$ has the Gauss property. Then the Laurent series
  expansion of $F$ with respect to any vertex $\boldsymbol{v}$ of $N (Q)$ is
  supported on a proper cone, namely the cone generated by $N (Q
  /\boldsymbol{x}^{\boldsymbol{v}})$.
\end{corollary}

\begin{proof}
  Recall from \eqref{eq:laurent:Q} that $\boldsymbol{x}^{\boldsymbol{v}} / Q$ has
  a Laurent series supported on the proper cone $C$ generated by the vectors
  $N (Q /\boldsymbol{x}^{\boldsymbol{v}})$. By Proposition~\ref{prop:NPNQ}, $N (P)
  \subseteq N (Q)$, so that the Laurent polynomial $P
  /\boldsymbol{x}^{\boldsymbol{v}}$ is supported on $N (P
  /\boldsymbol{x}^{\boldsymbol{v}}) \subset C$. Hence, multiplying the Laurent
  series for $\boldsymbol{x}^{\boldsymbol{v}} / Q$ with $P
  /\boldsymbol{x}^{\boldsymbol{v}}$, results in a Laurent series for $P / Q$
  supported on $C$.
\end{proof}

A face of $N (Q)$ is a nonempty set $F \subseteq N (Q)$, which is the
intersection of $N (Q)$ and $h (\boldsymbol{w}) = d$, where $h$ is a linear form
such that $N (Q)$ is contained in the half-space $h (\boldsymbol{w}) \geq
d$. Let $F$ be a face of $N (Q)$. We denote with $Q_F$ the Laurent polynomial,
which is the sum of monomials of $Q$ with support in $F$. $P_F$ is likewise
obtained from $P$.

\begin{proposition}
  \label{prop:PQF}Let $P, Q \in \mathbb{Z} [\boldsymbol{x}^{\pm 1}]$. If $P /
  Q$ satisfies the Gauss congruences for the prime $p$, then so does $P_F /
  Q_F$ for every face $F$ of $N (Q)$.
\end{proposition}

\begin{proof}
  Choose a vertex $\boldsymbol{v}$ of $N (Q)$ contained in $F$. After
  multiplication of $P, Q$ with $\boldsymbol{x}^{-\boldsymbol{v}}$, we may as well
  assume that $\boldsymbol{v}=\boldsymbol{0}$. Let $h$ be the linear form such
  that $F$ is given as the intersection of $h (\boldsymbol{w}) = 0$ and $N (Q)$.
  Let $C$ be the cone spanned by the vectors of $N (Q)$.
  
  Let $\sum_{\boldsymbol{k}} f_{\boldsymbol{k}} \boldsymbol{x}^{\boldsymbol{k}}$ be
  the Laurent series for $P / Q$ with respect to $\boldsymbol{0}$. We observe
  from \eqref{eq:laurent:Q} that the Laurent expansion of $1 / Q_F$ with
  respect to $\boldsymbol{0}$ is obtained from the corresponding expansion of $1
  / Q$ by selecting only those terms corresponding to $\boldsymbol{k}$ such that
  $h (\boldsymbol{k}) = 0$. Consequently, the Laurent expansion of $P_F / Q_F$
  with respect to $\boldsymbol{0}$ is similarly given by
  \begin{equation*}
    \sum_{\boldsymbol{k} \in \mathbb{Z}^n} g_{\boldsymbol{k}}
     \boldsymbol{x}^{\boldsymbol{k}} = \sum_{\boldsymbol{k} \in \mathbb{Z}^n, h
     (\boldsymbol{k}) = 0} f_{\boldsymbol{k}} \boldsymbol{x}^{\boldsymbol{k}} .
  \end{equation*}
  Since $h$ is linear, the Gauss congruences $g_{\boldsymbol{m}p^r} \equiv
  g_{\boldsymbol{m}p^{r - 1}} \pmod{p^r}$, for $\boldsymbol{m}$
  such that $h (\boldsymbol{m}) = 0$, translate into $f_{\boldsymbol{m}p^r} \equiv
  f_{\boldsymbol{m}p^{r - 1}} \pmod{p^r}$ which holds by
  assumption. On the other hand, we trivially have $g_{\boldsymbol{m}p^r} \equiv
  g_{\boldsymbol{m}p^{r - 1}} \pmod{p^r}$ if $h
  (\boldsymbol{m}) \neq 0$ because then both $g_{\boldsymbol{m}p^r}$ and
  $g_{\boldsymbol{m}p^{r - 1}}$ are zero.
\end{proof}

\section{Determinants of logarithmic derivatives}\label{sec:det}

This section is concerned with a proof of Theorem~\ref{thm:det:intro}. As
indicated in Question~\ref{q:det} and the comments following it,
Theorem~\ref{thm:det:intro} appears to play a central role in the quest of
classifying rational functions with the Gauss property. In this section, we
therefore work over more general rings before again specializing to
$\mathbb{Z}$. Though possible, for other, less central, results in this paper
we do not pursue this level of generality.

\begin{definition}
  \label{def:frob}An integral domain $R$ with characteristic $0$ has a
  {\emph{$p$-Frobenius lift $\phi$}} if
  \begin{enumerate}
    \item $(p)$ is a prime ideal in $R$, and
    
    \item there is a ring homomorphism $\phi : R \rightarrow R$ such that
    $\phi (a) \equiv a^p \pmod{p}$ for every $a \in R$.
  \end{enumerate}
  In the sequel, we often write $a^{\phi}$ instead of $\phi (a)$.
\end{definition}

The most common example we shall look at is $\mathbb{Z}$, in which case the
identity map $\phi$ is a $p$-Frobenius lift for all primes $p$. More
generally, if $N \in \mathbb{Z}$, we can consider the ring $\mathbb{Z} [1 /
N]$, in which case the identity map $\phi$ is a $p$-Frobenius lift for all
primes $p$ not dividing $N$. A nontrivial example is the polynomial ring $R
=\mathbb{Z} [x]$ with $\phi (Q (x)) = Q (x^p)$. The following lemma suggests
a generalization of the Gauss congruences to $R$.

\begin{lemma}
  \label{lem:powers:frob}Let $R$ be a domain with $p$-Frobenius lift $\phi$.
  Then, for any $a \in R$ and any positive integer
  \begin{equation}
    a^{m p^r} \equiv (a^{\phi})^{m p^{r - 1}} \quad (\operatorname{mod} p^r)
    \label{eq:gauss:R}
  \end{equation}
  for all integers $m \geq 0$ and $r \geq 1$.
\end{lemma}

\begin{proof}
  It suffices to prove the statement for $m = 1$. We use induction on $r$. For
  $r = 1$, the statement follows from the definition of $\phi$. Suppose that
  \eqref{eq:gauss:R} is true for some $r \geq 1$. That is, $a^{p^r} =
  (a^{\phi})^{p^{r - 1}} + p^r b$ for some $b \in R$. Raise this equality to
  the $p$th power and consider the result modulo $p^{r + 1}$. Using $p^{k r}
  \equiv 0 \pmod{p^{r + 1}}$ for all $k > 1$, we obtain
  \begin{equation*}
    a^{p^{r + 1}} \equiv (a^{\phi})^{p^r} + p (a^{\phi})^{p^{r - 1} (p - 1)}
     p^r b \equiv (a^{\phi})^{p^r} \quad (\operatorname{mod} p^{r + 1}) .
  \end{equation*}
  This completes the induction step.
\end{proof}

Extending Definition~\ref{def:gauss:laurent}, we therefore say that a Laurent
series $f = \sum_{\boldsymbol{k} \in \mathbb{Z}^n} f_{\boldsymbol{k}}
\boldsymbol{x}^{\boldsymbol{k}} \in R [[\boldsymbol{x}^{\pm 1}]]$ satisfies the
{\emph{Gauss congruences}} for the prime $p$ if $R$ has a $p$-Frobenius lift
$\phi$ and
\begin{equation*}
  f_{\boldsymbol{m}p^r} \equiv f_{\boldsymbol{m}p^{r - 1}}^{\phi} \quad
   (\operatorname{mod} p^r)
\end{equation*}
for all $\boldsymbol{m} \in \mathbb{Z}^n$ and all $r \geq 1$.

\begin{proposition}
  \label{prop:product}Let $R$ be an integral domain and let $f \in R
  [[\boldsymbol{x}^{\pm 1}]]$ with constant term $1$. If $\operatorname{supp} (f)
  \subseteq C$ for a proper cone $C$, then there exist $a_{\boldsymbol{k}} \in
  R$ such that
  \begin{equation}
    f = \prod_{\boldsymbol{k} \in C, \boldsymbol{k} \neq \boldsymbol{0}} (1 -
    a_{\boldsymbol{k}} \boldsymbol{x}^{\boldsymbol{k}}) . \label{eq:product:f}
  \end{equation}
\end{proposition}

\begin{proof}
  Choose $\alpha$ to be a linear form such that $\alpha (\boldsymbol{w}) > 0$
  for all nonzero $\boldsymbol{w} \in C$, with the additional property that
  $\alpha (\boldsymbol{w}) \in \mathbb{Z}$ for $\boldsymbol{w} \in
  \mathbb{Z}^n$. For $r \in \mathbb{Z}_{> 0}$, let $I_r$ be the ideal
  consisting of Laurent series
  \begin{equation*}
    \sum_{\boldsymbol{k} \in C, \alpha (\boldsymbol{k}) \geq r}
     g_{\boldsymbol{k}} \boldsymbol{x}^{\boldsymbol{k}} .
  \end{equation*}
  We now construct the exponents $a_{\boldsymbol{k}}$ using induction on $\alpha
  (\boldsymbol{k})$. Initialization follows from the observation that $f \equiv
  1 \pmod{I_1}$. Suppose we have constructed
  $a_{\boldsymbol{k}} \in R$, for all $\boldsymbol{k} \in C$ with $\alpha
  (\boldsymbol{k}) < r$, such that
  \begin{equation}
    f \equiv \prod_{\boldsymbol{k} \in C, \boldsymbol{k} \neq \boldsymbol{0}, \alpha
    (\boldsymbol{k}) < r} (1 - a_{\boldsymbol{k}} \boldsymbol{x}^{\boldsymbol{k}})
    \quad (\operatorname{mod} I_r) . \label{eq:product:I}
  \end{equation}
  Consider the Laurent series
  \begin{equation*}
    g = f \prod_{\boldsymbol{k} \in C, \boldsymbol{k} \neq \boldsymbol{0}, \alpha
     (\boldsymbol{k}) < r} (1 - a_{\boldsymbol{k}} \boldsymbol{x}^{\boldsymbol{k}})^{-
     1} = \sum_{\boldsymbol{k} \in C} g_{\boldsymbol{k}}
     \boldsymbol{x}^{\boldsymbol{k}},
  \end{equation*}
  and note that $g \equiv 1 \pmod{I_r}$. For $\boldsymbol{k}
  \in C$ with $\alpha (\boldsymbol{k}) = r$, we now simply choose
  $a_{\boldsymbol{k}} = g_{\boldsymbol{k}}$. By construction, \eqref{eq:product:I}
  then holds with $r$ replaced by $r + 1$.
\end{proof}

Recall that $\theta_i = x_i  \frac{\partial}{\partial x_i}$ denotes the Euler
operator. Suppose that $f \in R [[\boldsymbol{x}^{\pm 1}]]$ is a Laurent series
satisfying the conditions of Proposition~\ref{prop:product}, so that
\eqref{eq:product:f} holds. Then,
\begin{equation*}
  \frac{\theta_i f}{f} = - \sum_{\boldsymbol{k} \in C, \boldsymbol{k} \neq
   \boldsymbol{0}} \frac{k_i a_{\boldsymbol{k}} \boldsymbol{x}^{\boldsymbol{k}}}{1 -
   a_{\boldsymbol{k}} \boldsymbol{x}^{\boldsymbol{k}}} \in R [[\boldsymbol{x}^{\pm
   1}]] .
\end{equation*}
If $R$ is a domain with $p$-Frobenius lift $\phi$, then a brief argument and
Lemma~\ref{lem:powers:frob} show that $k_i a_{\boldsymbol{k}}
\boldsymbol{x}^{\boldsymbol{k}} / (1 - a_{\boldsymbol{k}}
\boldsymbol{x}^{\boldsymbol{k}})$ satisfies the Gauss congruences for $p$. We
conclude that $\theta_i f / f$, as a sum of such terms, satisfies the Gauss
congruences for $p$. This is the case $m = 1$ of the next result, which
generalizes this observation.

\begin{theorem}
  \label{thm:det}Let $R$ be a domain with $p$-Frobenius lift $\phi$. Let $m
  \leq n$ and let $f_1, \ldots, f_m \in R [[\boldsymbol{x}^{\pm 1}]]$, each
  with constant term equal to $1$ and $\operatorname{supp} (f_j) \subseteq C$ for a
  proper cone $C$. Then the Laurent series
  \begin{equation}
    F = \det \left(\frac{\theta_i f_j}{f_j} \right)_{i, j = 1, \ldots, m}
    \label{eq:det}
  \end{equation}
  satisfies the Gauss congruences for $p$.
\end{theorem}

\begin{proof}
  Suppose that $m < n$. If we define $f_i = x_i$ for $i = m + 1, m + 2,
  \ldots, n$, then the original $m \times m$ determinant \eqref{eq:det}
  obtained from $f_1, \ldots, f_m$ is the same as the $n \times n$ determinant
  obtained from $f_1, \ldots, f_n$. In the sequel, we may therefore assume
  that $m = n$.
  
  We start with the special case $f_j = 1 - a_j
  \boldsymbol{x}^{\boldsymbol{k}^{(j)}}$ with $\boldsymbol{k}^{(j)} = (k^{(j)}_1,
  \ldots, k^{(j)}_n) \in C$ and $a_j \in R$. Let $K$ be the $n \times n$
  matrix with entries $k_i^{(j)}$ with $i, j = 1, 2, \ldots, n$. Then,
  \begin{equation*}
    F = \det \left(\frac{- k^{(j)}_i a_j
     \boldsymbol{x}^{\boldsymbol{k}^{(j)}}}{1 - a_j
     \boldsymbol{x}^{\boldsymbol{k}^{(j)}}} \right)_{i, j = 1, \ldots, n} = (-
     1)^n \det (K) \prod_{j = 1}^n \frac{a_j
     \boldsymbol{x}^{\boldsymbol{k}^{(j)}}}{1 - a_j
     \boldsymbol{x}^{\boldsymbol{k}^{(j)}}},
  \end{equation*}
  which can be expanded as
  \begin{equation*}
    F = (- 1)^n \det (K) \sum_{r_1, \ldots, r_n \geq 1}^n \left(\prod_{j = 1}^n a_j^{r_j} \right) \boldsymbol{x}^{r_1 \boldsymbol{k}^{(1)} +
     \cdots + r_n \boldsymbol{k}^{(n)}} = \sum_{\boldsymbol{k} \in C}
     c_{\boldsymbol{k}} \boldsymbol{x}^{\boldsymbol{k}} .
  \end{equation*}
  If $\det (K) = 0$, then all terms are zero and the claim is trivially true.
  So, let us assume $\det (K) \neq 0$. Then the exponents $r_1
  \boldsymbol{k}^{(1)} + \cdots + r_n \boldsymbol{k}^{(n)} = K\boldsymbol{r}$ are in
  one-to-one correspondence with the $n$-tuples $\boldsymbol{r}= (r_1, \ldots,
  r_n)$.
  
  First, suppose that $\boldsymbol{r}$ is such that $p$ divides $\boldsymbol{r}$.
  Then, by Lemma~\ref{lem:powers:frob},
  \begin{equation*}
    \boldsymbol{a}^{\boldsymbol{r}} = \prod_{j = 1}^n a_j^{r_j} \equiv \prod_{j =
     1}^n (a_j^{\phi})^{r_j / p} = (\boldsymbol{a}^{\phi})^{\boldsymbol{r}/ p}
     \pmod{p^{\nu_p (\boldsymbol{r}}}),
  \end{equation*}
  so that
  \begin{eqnarray*}
    c_{K\boldsymbol{r}} & = & (- 1)^n \det (K) \boldsymbol{a}^{\boldsymbol{r}}\\
    & \equiv & (- 1)^n \det (K) (\boldsymbol{a}^{\phi})^{\boldsymbol{r}/ p}\\
    & = & ((- 1)^n \det (K) \boldsymbol{a}^{\boldsymbol{r}/ p})^{\phi} =
    c_{K\boldsymbol{r}/ p}^{\phi} \quad (\operatorname{mod} p^{\nu_p (\boldsymbol{r}) +
    \nu_p (\det (K))}) .
  \end{eqnarray*}
  Some linear algebra shows that $\nu_p (K\boldsymbol{r}) \leq \nu_p (\det
  (K)) + \nu_p (\boldsymbol{r})$. Therefore, $c_{K\boldsymbol{r}} \equiv
  c_{K\boldsymbol{r}/ p}^{\phi} \pmod{p^{\nu_p
  (K\boldsymbol{r}}})$. Next, suppose that $p$ does not divide $\boldsymbol{r}$.
  Then $\nu_p (K\boldsymbol{r}) \leq \nu_p (\det (K))$ and, thus,
  $c_{K\boldsymbol{r}} = (- 1)^n \det (K) \boldsymbol{a}^{\boldsymbol{r}}$ is
  divisible by $p^{\nu_p (K\boldsymbol{r})}$, whereas $c_{K\boldsymbol{r}/ p} =
  0$. Hence, the congruence $c_{K\boldsymbol{r}} \equiv c_{K\boldsymbol{r}/
  p}^{\phi} \pmod{p^{\nu_p (K\boldsymbol{r}}})$ holds again.
  We conclude that $F$ satisfies the Gauss congruences for the prime $p$.
  
  In the case of general Laurent series $f_j$, we consider the differential
  form
  \begin{equation*}
    \Omega = \frac{\md f_1}{f_1} \wedge \cdots \wedge \frac{\md
     f_n}{f_n}
  \end{equation*}
  and observe that the coefficient of $\md x_1 \wedge \cdots \wedge \md
  x_n / (x_1 \cdots x_n)$ is given by $F$. On the other hand, it follows from
  Proposition~\ref{prop:product} that
  \begin{equation*}
    f_j = \prod_{\boldsymbol{k} \in C, \boldsymbol{k} \neq \boldsymbol{0}} (1 -
     a^{(j)}_{\boldsymbol{k}} \boldsymbol{x}^{\boldsymbol{k}}),
  \end{equation*}
  for $a^{(j)}_{\boldsymbol{k}} \in R$, so that
  \begin{equation*}
    \Omega = (- 1)^n \sum_{\substack{
       \boldsymbol{k}^{(1)}, \ldots, \boldsymbol{k}^{(n)} \in C\\
       \boldsymbol{k}^{(1)}, \ldots, \boldsymbol{k}^{(n)} \neq \boldsymbol{0}
     }} \left(\prod_{j = 1}^n a^{(j)}_{\boldsymbol{k}} \right)
     \frac{\md \boldsymbol{x}^{\boldsymbol{k}^{(1)}}}{1 -
     a^{(1)}_{\boldsymbol{k}} \boldsymbol{x}^{\boldsymbol{k}^{(1)}}} \wedge \cdots
     \wedge \frac{\md \boldsymbol{x}^{\boldsymbol{k}^{(n)}}}{1 -
     a^{(n)}_{\boldsymbol{k}} \boldsymbol{x}^{\boldsymbol{k}^{(n)}}} .
  \end{equation*}
  From our initial special case, we see that the coefficient of $\md x_1
  \wedge \cdots \wedge \md x_n / (x_1 \cdots x_n)$ in each term satisfies
  the Gauss congruences for $p$. Hence, the same holds for their sum, which
  equals $F$.
\end{proof}

\begin{theorem}
  \label{thm:det:rat}Let $m \leq n$ and let $f_1, \ldots, f_m \in
  \mathbb{Q} (\boldsymbol{x})$ be nonzero. Then the rational function
  \begin{equation}
    \det \left(\frac{\theta_i f_j}{f_j} \right)_{i, j = 1, \ldots, m}
    \label{eq:det:rat}
  \end{equation}
  has the Gauss property.
\end{theorem}

\begin{proof}
  Let us write $D (f_1, \ldots, f_m)$ for \eqref{eq:det:rat}. Observe that
  this quantity is logarithmic in each of its arguments $f_j$. That is, for
  instance,
  \begin{equation*}
    D (g h, f_2, \ldots, f_m) = D (g, f_2, \ldots, f_m) + D (h, f_2, \ldots,
     f_m) .
  \end{equation*}
  From this logarithmic property it follows that it suffices to restrict to
  Laurent polynomials $f_1, \ldots, f_m \in \mathbb{Z} [\boldsymbol{x}^{\pm
  1}]$ (we could further restrict to polynomials but the argument to follow
  works with Laurent polynomials).
  
  We next determine vertices $\boldsymbol{v}_i$ of $N (f_i)$ such that $f_i
  /\boldsymbol{x}^{\boldsymbol{v}_i}$ are Laurent polynomials with support in the
  same proper cone $C$. To that end, let $\alpha$ be a linear form on
  $\mathbb{R}^n$ whose coefficients are $\mathbb{Q}$-linearly independent.
  For each $i = 1, \ldots, m$, let $c_i$ be the minimum of the set $\{ \alpha
  (\boldsymbol{x}) : \boldsymbol{x} \in N (f_i) \}$. Then $N (f_i)$ is contained
  in the half-space $\alpha (\boldsymbol{x}) \geq c_i$. Moreover, because
  the coefficients of $\alpha$ are $\mathbb{Q}$-linearly independent, the
  hyperplane $\alpha (\boldsymbol{x}) = c_i$ intersects $N (f_i)$ in a unique
  vertex $\boldsymbol{v}_i$. Hence, $N (f_i /\boldsymbol{x}^{\boldsymbol{v}_i})$ is
  contained in the half-space $\alpha (\boldsymbol{x}) \geq 0$, and the
  intersection of $N (f_i /\boldsymbol{x}^{\boldsymbol{v}_i})$ and $\alpha
  (\boldsymbol{x}) = 0$ consists of only the point $\boldsymbol{0}$. Let $C$ be
  the cone spanned by all the $N (f_i /\boldsymbol{x}^{\boldsymbol{v}_i})$. Note
  that $C$ is proper because $\alpha (\boldsymbol{x}) > 0$ for all nonzero
  $\boldsymbol{x} \in C$.
  
  By construction, the Laurent polynomials $g_i = f_i
  /\boldsymbol{x}^{\boldsymbol{v}_i}$ are supported on $C$ and $\boldsymbol{0}$ is a
  vertex of $N (g_i)$. Using the logarithmic property of \eqref{eq:det:rat},
  and the fact that $\partial x_j / \partial x_i = 0$ if $i \neq j$, it
  follows that $D (f_1, \ldots, f_m)$ is a $\mathbb{Z}$-linear combination of
  terms of the form $D (g_{i_1}, \ldots, g_{i_s})$ with $1 \leq i_1 <
  \ldots < i_s \leq m$. On the other hand, it follows from
  Theorem~\ref{thm:det} that $D (g_{i_1}, \ldots, g_{i_s})$ has the Gauss
  property, from which we conclude that $D (f_1, \ldots, f_m)$ has the Gauss
  property as well.
\end{proof}

\section{A classification result}\label{sec:linear}

We begin with a somewhat technical but general result. Observe that the case
$r = 0$ implies that $q_{\boldsymbol{k}} \boldsymbol{x}^{\boldsymbol{k}} / Q$ in
\eqref{eq:Qdet} satisfies Gauss congruences. Likewise, the determinant in
\eqref{eq:Qdet} satisfies Gauss congruences by Theorem~\ref{thm:det}. That is,
both factors of \eqref{eq:Qdet} satisfy Gauss congruences. However, if two
Laurent series satisfy Gauss congruences, then it is not the case, in general,
that their product satisfies Gauss congruences as well.

\begin{proposition}
  \label{prop:Qdet}Let $Q \in \mathbb{Z}_p [[\boldsymbol{x}, \boldsymbol{y}]]$
  with constant term $1$. Suppose that $Q$ is linear in the variables
  $\boldsymbol{x}= x_1, \ldots, x_n$ (but not necessarily in $\boldsymbol{y}= y_1,
  \ldots, y_m$). Further, for $0 \leq r \leq m$, let $f_1, \ldots,
  f_r \in \mathbb{Z}_p [[\boldsymbol{y}]]$ with constant term $1$. Write $Q =
  \sum_{\boldsymbol{k}} q_{\boldsymbol{k}} (\boldsymbol{y})
  \boldsymbol{x}^{\boldsymbol{k}}$. Then, for any $\boldsymbol{k}$,
  \begin{equation}
    \frac{q_{\boldsymbol{k}} \boldsymbol{x}^{\boldsymbol{k}}}{Q} \det \left(\frac{\theta_i f_j}{f_j} \right)_{i, j = 1, \ldots, r} \label{eq:Qdet}
  \end{equation}
  satisfies Gauss congruences for the prime $p$.
\end{proposition}

\begin{proof}
  Let $\boldsymbol{k}= (k_1, \ldots, k_n) \in \{ 0, 1 \}^n$. Notice that
  \begin{equation*}
    \left[ \prod_{j = 1}^n \theta_j^{k_j} (1 - \theta_j)^{1 - k_j} \right] Q
     = q_{\boldsymbol{k}} \boldsymbol{x}^{\boldsymbol{k}} .
  \end{equation*}
  Hence, it suffices to show that \eqref{eq:Qdet} holds with $q_{\boldsymbol{k}}
  \boldsymbol{x}^{\boldsymbol{k}}$ replaced with any product of $\theta_j$ applied
  to $Q$. Without loss of generality, we consider the product $\theta_1 \cdots
  \theta_{\ell} Q$.
  
  For $j = 1, \ldots, \ell$, define
  \begin{equation*}
    g_j = \frac{\partial}{\partial x_{j - 1}} \cdots \frac{\partial}{\partial
     x_2} \frac{\partial}{\partial x_1} Q,
  \end{equation*}
  so that $g_1 = Q$. Note that $g_j$ does not depend on the variables $x_1,
  \ldots, x_{j - 1}$ because $Q$ is linear in the variables $x_1, \ldots,
  x_n$. Therefore, applying Theorem~\ref{thm:det} to $g_1, \ldots, g_{\ell},
  f_1, \ldots, f_r$, we find that
  \begin{equation*}
    \frac{\theta_1 g_1}{g_1} \cdots \frac{\theta_{\ell} g_{\ell}}{g_{\ell}}
     \det \left(\frac{\theta_i f_j}{f_j} \right)_{i, j = 1, \ldots, r}
  \end{equation*}
  satisfies the Gauss congruences for $p$. Observe that, since $\theta_j g_j =
  x_j g_{j + 1}$,
  \begin{equation*}
    \frac{\theta_1 g_1}{g_1} \cdots \frac{\theta_{\ell} g_{\ell}}{g_{\ell}} =
     \frac{\theta_1 \cdots \theta_{\ell} Q}{Q} .
  \end{equation*}
  We have therefore shown that \eqref{eq:Qdet} indeed holds with
  $q_{\boldsymbol{k}} \boldsymbol{x}^{\boldsymbol{k}}$ replaced with any product of
  $\theta_j$ applied to $Q$.
\end{proof}

\begin{example}
  Take $Q (x, y) = \frac{1}{f (y)} - x$ for some
  \begin{equation*}
    f = \sum_{k = 0}^{\infty} a_k y^k \in \mathbb{Z} [[y]]^{\times} .
  \end{equation*}
  Since $Q (x, y)$ is linear in $x$, Proposition~\ref{prop:Qdet} implies that
  \begin{equation*}
    \frac{1 / f (y)}{Q (x, y)} = \frac{1}{1 - f (y) x} = \sum_{k =
     0}^{\infty} f (y)^k x^k
  \end{equation*}
  satisfies the Gauss congruences for all primes. Equivalently, the
  coefficients $c_{k, \ell}$ of $x^k y^{\ell}$, that is
  \begin{equation*}
    c_{k, \ell} = \sum_{\substack{
       j_1, \ldots, j_k \geq 0\\
       j_1 + \ldots + j_k = \ell
     }} a_{j_1} \cdots a_{j_{\ell}},
  \end{equation*}
  satisfy the congruences $c_{k p^r, \ell p^r} \equiv c_{k p^{r - 1}, \ell
  p^{r - 1}} \pmod{p^r}$ for all primes $p$ and all $r
  \geq 1$. Establishing these congruences directly is less
  straightforward.
\end{example}

\begin{example}
  \label{eg:lucas:xy}Take $Q (x, y) = 1 - x - x^2 - y$. Since $Q$ is linear in
  $y$, it follows from Proposition~\ref{prop:Qdet} with
  $\boldsymbol{k}=\boldsymbol{0}$ that, for any $f \in \mathbb{Z}_p [[x]]$ with
  $f (0) = 1$, the product
  \begin{equation*}
    \frac{1 - x - x^2}{1 - x - x^2 - y}  \frac{\theta_x f}{f}
  \end{equation*}
  satisfies the Gauss congruences for $p$. In fact, since $(\theta_x f) / f$
  is logarithmic in $f$ and since $(\theta_x x^k) / x^k = k$, the same is
  obviously true for any $f \in \mathbb{Z}_p [[x]]$. For instance, choosing
  $f = x^2 / (1 - x - x^2)$, we find that
  \begin{equation*}
    \frac{2 - x}{1 - x - x^2 - y}
  \end{equation*}
  satisfies the Gauss congruences for all primes. In this particular case, the
  same conclusion follows from Theorem~\ref{thm:mostlylinear}.
\end{example}

\begin{theorem}
  \label{thm:mostlylinear}Let $P, Q \in \mathbb{Z} [z, \boldsymbol{x}]$ such
  that $Q$ is linear in the variables $\boldsymbol{x}$. Write $P =
  \sum_{\boldsymbol{k}} p_{\boldsymbol{k}} (z) \boldsymbol{x}^{\boldsymbol{k}}$ and $Q
  = \sum_{\boldsymbol{k}} q_{\boldsymbol{k}} (z) \boldsymbol{x}^{\boldsymbol{k}}$ with
  $p_{\boldsymbol{k}}, q_{\boldsymbol{k}} \in \mathbb{Z} [z]$. Then $P / Q$ has
  the Gauss property if and only if $p_{\boldsymbol{k}} \neq 0$ implies
  $q_{\boldsymbol{k}} \neq 0$ and $p_{\boldsymbol{k}} / q_{\boldsymbol{k}}$ has the
  Gauss property for all $\boldsymbol{k}$ with $q_{\boldsymbol{k}} \neq 0$.
\end{theorem}

\begin{proof}
  Suppose $P / Q$ has the Gauss property. Let $\boldsymbol{k}$ be such that
  $p_{\boldsymbol{k}} \neq 0$. Suppose $q_{\boldsymbol{k}} = 0$. Since the points
  $\boldsymbol{k}$ are vertices of the hypercube $[0, 1]^n$, this means that the
  support of $p_{\boldsymbol{k}} \boldsymbol{x}^{\boldsymbol{k}}$ is outside $N
  (Q)$, contradicting Proposition~\ref{prop:NPNQ}. Hence, $q_{\boldsymbol{k}}
  \neq 0$. Further, notice that $p_{\boldsymbol{k}} / q_{\boldsymbol{k}}$ is $P_F
  / Q_F$, where $F$ is the face of $N (Q)$ corresponding to $\boldsymbol{k}$.
  According to Proposition~\ref{prop:PQF}, $p_{\boldsymbol{k}} /
  q_{\boldsymbol{k}}$ has the Gauss property.
  
  Now, fix $\boldsymbol{k}$ and suppose that $q_{\boldsymbol{k}} \neq 0$ and that
  $p_{\boldsymbol{k}} / q_{\boldsymbol{k}}$ has the Gauss property. We shall prove
  that $p_{\boldsymbol{k}} (z) \boldsymbol{x}^{\boldsymbol{k}} / Q$ has the Gauss
  property. The general theorem then follows by summing over all such
  $\boldsymbol{k}$. First, observe that Theorem~\ref{thm:minton} tells us that
  there are polynomials $u_j \in \mathbb{Z} [z]$ such that $p_{\boldsymbol{k}}
  / q_{\boldsymbol{k}} = \sum_j c_j z u_j' / u_j$ for some $c_j \in
  \mathbb{Q}$.
  
  By Proposition~\ref{prop:Qdet},
  \begin{equation*}
    \sum_j \frac{q_{\boldsymbol{k}} (z) \boldsymbol{x}^{\boldsymbol{k}}}{Q} 
     \frac{z u_j'}{u_j} = \frac{q_{\boldsymbol{k}} (z)
     \boldsymbol{x}^{\boldsymbol{k}}}{Q}  \frac{p_{\boldsymbol{k}}
     (z)}{q_{\boldsymbol{k}} (z)} = \frac{p_{\boldsymbol{k}} (z)
     \boldsymbol{x}^{\boldsymbol{k}}}{Q}
  \end{equation*}
  has the Gauss property.
\end{proof}

\begin{example}
  Monthly problem \texttt{\#11757} \cite{monthly-11757}, proposed by
  Gessel, concerns the rational function
  \begin{equation*}
    F (x, y) = \frac{1}{(1 - 3 x) (1 - y - 3 x + 3 x^2)} = \sum_{m, n =
     0}^{\infty} c_{m, n} x^m y^n .
  \end{equation*}
  The problem asks the reader to show that the diagonal Taylor coefficients
  $c_{n, n}$ equal $9^n$. We will not spoil the fun of that challenge but only
  note that, as a consequence, the sequence of diagonal coefficients satisfies
  Gauss congruences for all primes. On the other hand, it is an immediate
  consequence of Theorem~\ref{thm:mostlylinear} and Theorem~\ref{thm:minton},
  applied to
  \begin{equation*}
    \frac{1}{(1 - 3 x) (1 - 3 x + 3 x^2)} = \frac{\theta_x u}{u}, \quad u =
     \frac{x (1 - 3 x + 3 x^2)}{(1 - 3 x)^3},
  \end{equation*}
  that $F (x, y)$ has the Gauss property for all primes. In other words, all
  Taylor coefficients $c_{m, n}$ satisfy Gauss congruences for all primes.
\end{example}

A useful and immediate consequence of Theorem~\ref{thm:mostlylinear} is the
following characterization of rational functions, whose denominator is linear
in each variable (that is, the denominator $Q \in \mathbb{Z} [\boldsymbol{x}]$
has support $\operatorname{supp} (Q) \subseteq \{ 0, 1 \}^n$).

\begin{theorem}
  \label{thm:linear}Let $P, Q \in \mathbb{Z} [\boldsymbol{x}]$ and suppose that
  $Q$ is linear in each variable. Then $P / Q$ has the Gauss property if and
  only if $N (P) \subseteq N (Q)$.
\end{theorem}

\begin{example}
  The {\emph{Delannoy numbers}}
  \begin{equation*}
    D_{n_1, n_2} = \sum_{k = 0}^{\min (n_1, n_2)} \binom{n_1}{k} \binom{n_1 +
     n_2 - k}{n_1}
  \end{equation*}
  are the Laurent series coefficients of the rational function
  \begin{equation*}
    \frac{1}{1 - x - y - x y} = \sum_{k_1, k_2 = 0}^{\infty} D_{k_1, k_2}
     x^{k_1} y^{k_2}
  \end{equation*}
  with respect to the vertex $(0, 0)$. By Theorem~\ref{thm:linear}, each of
  the rational functions
  \begin{equation*}
    \frac{1}{1 - x - y - x y}, \quad \frac{x}{1 - x - y - x y}, \quad
     \frac{y}{1 - x - y - x y}, \quad \frac{x y}{1 - x - y - x y}
  \end{equation*}
  has the Gauss property. In fact, they satisfy the Gauss congruences for all
  primes. Consequently, for any prime $p$ and $\boldsymbol{\delta} \in \{ 0, 1
  \}^2$, the Delannoy numbers $D_{\boldsymbol{n}}$ satisfy the (shifted)
  congruences
  \begin{equation*}
    D_{\boldsymbol{m}p^r -\boldsymbol{\delta}} \equiv D_{\boldsymbol{m}p^{r - 1}
     -\boldsymbol{\delta}} \quad (\operatorname{mod} p^r)
  \end{equation*}
  for all $\boldsymbol{m} \in \mathbb{Z}_{> 0}^2$ and all $r \geq 1$.
\end{example}

\section{Toroidal substitutions}

A substitution of the form $x_i =\boldsymbol{y}^{\boldsymbol{a}_i}$, $i = 1,
\ldots, n$, with $\boldsymbol{a}_1, \ldots, \boldsymbol{a}_n \in \mathbb{Q}^m$
linearly independent, is called a {\emph{toroidal substitution}}. Let $A$ be
the $m \times n$ matrix with columns $\boldsymbol{a}_1, \ldots, \boldsymbol{a}_n$.
Note that, for any $\boldsymbol{k} \in \mathbb{Z}^n$,
$\boldsymbol{x}^{\boldsymbol{k}} =\boldsymbol{y}^{A\boldsymbol{k}}$. We therefore
write the toroidal substitution simply as $\boldsymbol{x}=\boldsymbol{y}^A$. The
next result shows that toroidal substitutions preserve the Gauss property.

\begin{proposition}
  \label{prop:toroidal}Let $f = \sum_{\boldsymbol{k} \in \mathbb{Z}^n}
  f_{\boldsymbol{k}} \boldsymbol{x}^{\boldsymbol{k}} \in \mathbb{Z}_p
  [[\boldsymbol{x}^{\pm 1}]]$. Let $A$ be an $m \times n$ matrix with linearly
  independent columns $\boldsymbol{a}_1, \ldots, \boldsymbol{a}_n \in
  \mathbb{Q}^m$, such that $A\boldsymbol{k} \in \mathbb{Z}^m$ for all
  $\boldsymbol{k} \in \operatorname{supp} (f)$. Define the Laurent series $g \in
  \mathbb{Z}_p [[\boldsymbol{y}^{\pm 1}]]$ by
  \begin{equation*}
    g (y_1, \ldots, y_m) = f (\boldsymbol{y}^{\boldsymbol{a}_1}, \ldots,
     \boldsymbol{y}^{\boldsymbol{a}_n}) = \sum_{\boldsymbol{k} \in \mathbb{Z}^n}
     f_{\boldsymbol{k}} \boldsymbol{y}^{A\boldsymbol{k}} .
  \end{equation*}
  Suppose that the prime $p$ is such that $A \in \mathbb{Z}_p^{m \times n}$
  and $A \pmod{p}$ has rank $n$. Then, $g$ satisfies the
  Gauss congruences for $p$ if and only if $f$ does.
\end{proposition}

\begin{proof}
  Since $\operatorname{rank} (A) = n$, the map $\mathbb{Z}^n \rightarrow
  \mathbb{Q}^m$, $\boldsymbol{k} \mapsto A\boldsymbol{k}$, is injective. Let us
  write $I$ for the image of this map restricted to $\mathbb{Z}^m$. That is,
  $I$ is the set of $\boldsymbol{m} \in \mathbb{Z}^m$ such that $\boldsymbol{m}=
  A\boldsymbol{k}$ for some $\boldsymbol{k} \in \mathbb{Z}^n$. In the sequel, we
  write $A^{- 1} \boldsymbol{m}=\boldsymbol{k}$ for that unique vector
  $\boldsymbol{k}$.
  
  We claim that $\nu_p (A\boldsymbol{k}) = \nu_p (\boldsymbol{k})$ for all
  $\boldsymbol{k} \in \mathbb{Z}^n$. Since $A \in \mathbb{Z}_p^{m \times n}$,
  we obviously have $\nu_p (A\boldsymbol{k}) \geq \nu_p (\boldsymbol{k})$. On
  the other hand, $A \pmod{p}$ has rank $n$, that is,
  $A\boldsymbol{k} \equiv \boldsymbol{0} \pmod{p}$ implies
  $\boldsymbol{k} \equiv \boldsymbol{0} \pmod{p}$. It follows
  inductively (or from the fact that $A \pmod{p^r}$ has
  rank $n$) that $A\boldsymbol{k} \equiv \boldsymbol{0} \pmod{p^r}$ implies $\boldsymbol{k} \equiv \boldsymbol{0} \pmod{p^r}$. Hence, $\nu_p (A\boldsymbol{k}) \leq \nu_p (\boldsymbol{k})$.
  Observe that, as a consequence, $\boldsymbol{m}p^r \in I$ if and only if
  $\boldsymbol{m} \in I$.
  
  By construction,
  \begin{equation*}
    g = \sum_{\boldsymbol{m} \in \mathbb{Z}^m} g_{\boldsymbol{m}}
     \boldsymbol{y}^{\boldsymbol{m}} = \sum_{\boldsymbol{k} \in \mathbb{Z}^n}
     f_{\boldsymbol{k}} \boldsymbol{y}^{A\boldsymbol{k}} = \sum_{\boldsymbol{m} \in I}
     f_{A^{- 1} \boldsymbol{m}} \boldsymbol{y}^{\boldsymbol{m}},
  \end{equation*}
  so that $g_{\boldsymbol{m}} = f_{A^{- 1} \boldsymbol{m}}$ if $\boldsymbol{m} \in
  I$ and $g_{\boldsymbol{m}} = 0$ otherwise. On the other hand,
  $f_{\boldsymbol{k}} = g_{A\boldsymbol{k}}$ for all $\boldsymbol{k}$ such that
  $A\boldsymbol{k} \in \mathbb{Z}^m$.
  
  Suppose that $g$ satisfies the Gauss congruences for $p$. Let $\boldsymbol{k}
  \in \mathbb{Z}^n$. If $A\boldsymbol{k} \in \mathbb{Z}^m$, then
  \begin{equation*}
    f_{\boldsymbol{k}p^r} \equiv g_{(A\boldsymbol{k}) p^r} \equiv
     g_{(A\boldsymbol{k}) p^{r - 1}} \equiv f_{\boldsymbol{k}p^{r - 1}}
     \pmod{p^r} .
  \end{equation*}
  If $A\boldsymbol{k} \not\in \mathbb{Z}^m$, then it follows from $A \in
  \mathbb{Z}_p^{m \times n}$ that $A\boldsymbol{k}p^r \not\in \mathbb{Z}^m$.
  Hence, $f_{\boldsymbol{k}p^r} = f_{\boldsymbol{k}p^{r - 1}} = 0$. Thus, $f$
  satisfies the Gauss congruences for $p$.
  
  Finally, suppose that $f$ satisfies the Gauss congruences for $p$. Let
  $\boldsymbol{m} \in \mathbb{Z}^m$. If $\boldsymbol{m} \in I$, then
  $\boldsymbol{m}= A\boldsymbol{k}$ for some $\boldsymbol{k} \in \mathbb{Z}^n$, and
  \begin{equation*}
    g_{\boldsymbol{m}p^r} = f_{\boldsymbol{k}p^r} \equiv f_{\boldsymbol{k}p^{r -
     1}} = g_{\boldsymbol{m}p^{r - 1}} \pmod{p^r} .
  \end{equation*}
  If $\boldsymbol{m} \not\in I$, then $\boldsymbol{m}p^r, \boldsymbol{m}p^{r - 1} \not\in
  I$, so that $g_{\boldsymbol{k}p^r} = g_{\boldsymbol{k}p^{r - 1}} = 0$.
  Consequently, $g$ satisfies the Gauss congruences for $p$.
\end{proof}

\begin{example}
  It follows from Proposition~\ref{prop:toroidal} that
  \begin{equation*}
    \frac{2 - x y}{1 - x y - x^2 y^2}
  \end{equation*}
  has the Gauss property if and only if $(2 - x) / (1 - x - x^2)$ does. The
  latter is the generating function for the Lucas numbers. That it has the
  Gauss property follows, for instance, from Theorem~\ref{thm:det:rat}.
\end{example}

\begin{example}
  As observed in Example~\ref{eg:lucas:xy}, the rational function $F (x, y) =
  (2 - x) / (1 - x - x^2 - y)$ has the Gauss property. However, $F (x, 1) = (x
  - 2) / (x + x^2)$ does not have the Gauss property because it violates the
  necessary condition of Proposition~\ref{prop:NPNQ}. This illustrates that
  the condition in Proposition~\ref{prop:toroidal} on the rank of $A$ cannot
  be dropped.
\end{example}

\begin{example}
  \label{eg:toroidal:rat}Consider $Q = 1 + y_1^3 y_2 y_3 + y_1 y_2 y_3^3 + 3
  y_1^2 y_2 y_3^2$. In that case, $N (Q)$ lies in a two-dimensional subspace
  of $\mathbb{R}^3$. Note that $(3, 1, 1)$ and $(1, 1, 3)$ are vertices of $N
  (Q)$, while $(2, 1, 2)$ is not. We can obtain $Q (\boldsymbol{y})$ from
  $\tilde{Q} (\boldsymbol{x}) = 1 + x_1^2 + x_2^2 + 3 x_1 x_2$ via the toroidal
  substitution
  \begin{equation*}
    x_1 = y_1^{3 / 2} y_2^{1 / 2} y_3^{1 / 2}, \quad x_2 = y_1^{1 / 2} y_2^{1
     / 2} y_3^{3 / 2} .
  \end{equation*}
  Let $p > 2$ be a primes. It follows from Proposition~\ref{prop:toroidal}
  that $1 / Q$ satisfies the Gauss congruences for $p$ if and only if $1 /
  \tilde{Q}$ satisfies the Gauss congruences for $p$.
\end{example}

\begin{remark}
  \label{rk:q:toroidal}Let $\boldsymbol{x}=\boldsymbol{y}^A$ be a toroidal
  substitution with invertible matrix $A \in \mathbb{Q}^{n \times n}$.
  Observe that, for any $f_1, \ldots, f_n$,
  \begin{equation*}
    \det \left(\frac{\theta_{y_i} f_j}{f_j} \right)_{i, j = 1, \ldots, n} =
     \det (A) \det \left(\frac{\theta_{x_i} f_j}{f_j} \right)_{i, j = 1,
     \ldots, n} .
  \end{equation*}
  This can be seen, for instance, by recalling that the left-hand side is the
  coefficient of $\md y_1 \wedge \cdots \wedge \md y_n / (y_1 \cdots
  y_n)$ in $\md f_1 \wedge \cdots \wedge \md f_n / (f_1 \cdots f_n)$,
  and realizing that $\md x_1 \wedge \cdots \wedge \md x_n / (x_1 \cdots
  x_n)$ and $\md y_1 \wedge \cdots \wedge \md y_n / (y_1 \cdots y_n)$
  only differ by a factor of $\det (A)$. We conclude that an invertible
  toroidal substitution does not affect the answer to Question~\ref{q:det}.
\end{remark}

\begin{example}
  \label{eg:q:degree2}Suppose that $Q \in \mathbb{Z} [x, y]$ has total degree
  $2$. We will show that, for any rational function $P / Q$ with $P \in
  \mathbb{Z} [x, y]$, the answer to Question~\ref{q:det} is affirmative. That
  is, $P / Q$ has the Gauss property if and only if $P / Q$ can be written as
  a $\mathbb{Q}$-linear combination of functions of the form
  \eqref{eq:det:intro}. By Theorem~\ref{thm:det:rat}, we only need to prove
  the ``only if'' part of that statement.
  
  First, observe that this is a consequence of our proof of
  Theorem~\ref{thm:mostlylinear} in the case that $Q$ is linear in at least
  one of the variables $x, y$. We may therefore assume that $(2, 0)$ and $(0,
  2)$ are vertices of $N (Q)$. Suppose that $(0, 0)$ is not a vertex of $N
  (Q)$. Then $Q = b x + c y + d x^2 + e x y + f y^2$, so that $Q / x^2 = b u +
  c u v + d + e v + f v^2$ in terms of the toroidal substitution $u = 1 / x$,
  $v = y / x$. Note that the latter is linear in $u$, so that, by
  Proposition~\ref{prop:toroidal}, we are reduced to a known case. We may
  therefore assume in the sequel that $N (Q)$ is the triangle with vertices
  $(0, 0), (2, 0), (0, 2)$. The number of lattice points in $N (Q)$ is $6$.
  
  Consider the vector space $V_Q$ of polynomials $P \in \mathbb{Z} [x, y]$
  such that $P / Q$ has the Gauss property. It follows from
  Proposition~\ref{prop:NPNQ} that $V_Q$ consists of polynomials of total
  degree at most $2$. In particular, $\dim V_Q \leq 6$. On the other
  hand, by Theorem~\ref{thm:det:rat}, $1$ as well as $x \frac{\partial
  Q}{\partial x} / Q$ and $y \frac{\partial Q}{\partial y} / Q$ have the Gauss
  property. Hence, $V_Q$ contains $Q$, $x \partial Q / \partial x$ and $y
  \partial Q / \partial y$, implying that $\dim V_Q \geq 3$. We will see
  below that $\dim V_Q$ can indeed take any value in $\{ 3, 4, 5, 6 \}$. Let
  $F$ be one of the three $1$-dimensional faces of $N (Q)$. By
  Proposition~\ref{prop:PQF}, if $P \in V_Q$, then $P_F / Q_F$ has the Gauss
  property.
  
  Note that, possibly after a toroidal substitution, $P_F / Q_F$ is a
  univariate rational function $p_F (x) / q_F (x)$ with $p_F, q_F \in
  \mathbb{Z} [x]$, $\deg q_F = 2$ and $q_F (0) \neq 0$. By the same arguments
  as above, the vector space $V_{q_F}$ of polynomials $p \in \mathbb{Z} [x]$,
  such that $p / q_F$ has the Gauss property, has dimension $2$ or $3$.
  Moreover, it follows from Theorem~\ref{thm:minton} that $\dim V_{q_F} = 3$
  if and only if $q_F (x)$ has two distinct rational roots.
  
  Let $F_x$ be the face with vertices $(0, 0)$ and $(2, 0)$. Likewise, let
  $F_y$ be the face with vertices $(0, 0)$, $(0, 2)$, and $F_{x y}$ the face
  with vertices $(2, 0), (0, 2)$.
  
  Suppose $q{}_{F_{x y}} (x)$ has two distinct rational roots. That is, $Q = a +
  b x + c y + d x^2 + e x y + f y^2$ and $d x^2 + e x y + f y^2 = d (x +
  \alpha y) (x + \beta y)$ for some $\alpha, \beta \in \mathbb{Q}$ with
  $\alpha \neq \beta$. Without loss, $d = 1$. Observe that $Q = \varepsilon +
  (x + \alpha y + \gamma) (x + \beta y + \delta)$, where $\gamma = (\alpha b -
  c) / (\alpha - \beta)$, $\delta = (c - \beta b) / (\alpha - \beta)$ and
  $\varepsilon = a - \gamma \delta$ are all rational.
  Theorem~\ref{thm:det:rat}, applied to $f_1 = x + \alpha y + \gamma$ and $f_2
  = Q$, shows that
  \begin{equation*}
    \frac{\theta_x f_1}{f_1}  \frac{\theta_y f_2}{f_2} - \frac{\theta_y
     f_1}{f_1}  \frac{\theta_x f_2}{f_2} = \frac{(\beta - \alpha) x y}{Q}
  \end{equation*}
  has the Gauss property. In particular, $x y / Q$ has the Gauss property. By
  using a toroidal substitution to translate to this case and applying
  Proposition~\ref{prop:toroidal}, we conclude that, for $m \in \{ x, y, x y
  \}$, if $q{}_{F_m} (x)$ has two distinct rational roots, then $m / Q$ has the
  Gauss property and, by Remark~\ref{rk:q:toroidal}, is a linear combination
  of functions of the form \eqref{eq:det:intro}.
  
  Let $M \subseteq \{ x, y, x y \}$ consist of those $m$ such that $q{}_{F_m}
  (x)$ has two distinct rational roots. For each $m \in \{ x, y, x y \}$ with
  $m \not\in M$, we obtain a linear constraint for $V_Q$ coming from the
  condition that $P{}_{F_m} / Q{}_{F_m}$ has the Gauss property. These $3 - | M |$
  constraints are linearly independent, so that $\dim V_Q \leq 6 - (3 - |
  M |) = 3 + | M |$. On the other hand, $V_Q$ contains $Q$, $x \partial Q /
  \partial x$, $y \partial Q / \partial y$ as well as $m$ for $m \in M$. Since
  these are linearly independent, we conclude that $\dim V_Q = 3 + | M |$ and
  that all $P \in V_Q$ are a linear combination of functions of the form
  \eqref{eq:det:intro}.
\end{example}

\section{Univariate substitutions}\label{sec:subst}

In order to prove Theorem~\ref{thm:subst:rat:intro}, we begin with the
following corresponding result for Laurent series. The statement is
necessarily more technical because conditions are needed to ensure that the
composition of series is well-defined.

\begin{theorem}
  \label{thm:subst}Let $\alpha$ be a linear form on $\mathbb{R}^{n + 1}$ such
  that $\alpha (1, 0, 0, \ldots, 0) > 0$. Let $f (z, \boldsymbol{x}) \in
  \mathbb{Z}_p [[z^{\pm 1}, \boldsymbol{x}^{\pm 1}]]$ such that $\alpha
  (\boldsymbol{w}) > 0$ for all $\boldsymbol{w} \in \operatorname{supp} (f)$. Let $g (z)
  \in z^r \mathbb{Z}_p [[z]]^{\times}$, for $r \in \mathbb{Z}_{> 0}$.
  Suppose $f$ satisfies the Gauss congruences for the prime $p$. Then so does
  \begin{equation*}
    F (z, \boldsymbol{x}) = \frac{z g' (z)}{g (z)} f (g (z), \boldsymbol{x}) .
  \end{equation*}
\end{theorem}

\begin{proof}
  By assumption,
  \begin{equation}
    f (z, \boldsymbol{x}) = \sum_{\alpha (\ell, \boldsymbol{k}) > 0} f_{\ell,
    \boldsymbol{k}} z^{\ell} \boldsymbol{x}^{\boldsymbol{k}}, \label{eq:subst:f}
  \end{equation}
  where the sum is over all $\ell \in \mathbb{Z}$, $\boldsymbol{k} \in
  \mathbb{Z}^n$ such that $\alpha (\ell, \boldsymbol{k}) > 0$. First, let us
  note that
  \begin{equation*}
    f (g (z), \boldsymbol{x}) = \sum_{\alpha (\ell, \boldsymbol{k}) > 0} f_{\ell,
     \boldsymbol{k}} g (z)^{\ell} \boldsymbol{x}^{\boldsymbol{k}}
  \end{equation*}
  is a well-defined Laurent series in $\mathbb{Z}_p [[z^{\pm 1},
  \boldsymbol{x}^{\pm 1}]]$, since contributions to the coefficient of $z^L
  \boldsymbol{x}^{\boldsymbol{k}}$ only come from the indices $(\ell,
  \boldsymbol{k})$ with $\ell \leq L$. By the assumption that $\alpha (1,
  \boldsymbol{0}) > 0$, there are only finitely many such $(\ell, \boldsymbol{k})$
  with $\alpha (\ell, \boldsymbol{k}) > 0$.
  
  Write $f (z, \boldsymbol{x}) = f_0 (\boldsymbol{x}) + z f_1 (z, \boldsymbol{x})$.
  It follows from the case $m = 1$ of Theorem~\ref{thm:det} that $\frac{z g'
  (z)}{g (z)} f_0 (\boldsymbol{x})$ satisfies the Gauss congruences. We may
  therefore replace $f (z, \boldsymbol{x})$ with $f (z, \boldsymbol{x}) - f_0
  (\boldsymbol{x})$. In other words, we may assume that the sum in
  \eqref{eq:subst:f} is over all $\ell \in \mathbb{Z}$, $\boldsymbol{k} \in
  \mathbb{Z}^n$ such that $\alpha (\ell, \boldsymbol{k}) > 0$ and $\ell \neq
  0$. All subsequent sums are assumed to be of this form. Observe that
  \begin{equation*}
    F (z, \boldsymbol{x}) = \frac{z g' (z)}{g (z)} f (g (z), \boldsymbol{x}) =
     \sum_{\ell \neq 0, \boldsymbol{k}} \frac{f_{\ell, \boldsymbol{k}}}{\ell} z
     \frac{\md}{\md z} [g (z)^{\ell}] \boldsymbol{x}^{\boldsymbol{k}} .
  \end{equation*}
  We rewrite this as $F = F_1 + F_2$ with
  \begin{equation*}
    F_1 (z, \boldsymbol{x}) = \sum_{\ell, \boldsymbol{k}} \frac{f_{\ell,
     \boldsymbol{k}} - f_{\ell / p, \boldsymbol{k}/ p}}{\ell} z
     \frac{\md}{\md z} [g (z)^{\ell}] \boldsymbol{x}^{\boldsymbol{k}}
  \end{equation*}
  and
  \begin{equation*}
    F_2 (z, \boldsymbol{x}) = \sum_{\ell, \boldsymbol{k}} \frac{f_{\ell / p,
     \boldsymbol{k}/ p}}{\ell} z \frac{\md}{\md z} [g (z)^{\ell}]
     \boldsymbol{x}^{\boldsymbol{k}},
  \end{equation*}
  where we use the convention that $f_{\ell / p, \boldsymbol{k}/ p} = 0$ if
  $\ell$ or $\boldsymbol{k}$ is not divisible by $p$. The second sum equals
  \begin{equation*}
    F_2 (z, \boldsymbol{x}) = \sum_{\ell, \boldsymbol{k}} f_{\ell / p,
     \boldsymbol{k}/ p} z \frac{\md}{\md z} \left[ \frac{g (z)^{\ell} - g
     (z^p)^{\ell / p}}{\ell} \right] \boldsymbol{x}^{\boldsymbol{k}} + F (z^p,
     \boldsymbol{x}^p) .
  \end{equation*}
  To prove our claim, we need to show that the coefficient of $z^L
  \boldsymbol{x}^{\boldsymbol{k}}$ in $F (z, \boldsymbol{x}) - F (z^p,
  \boldsymbol{x}^p)$ is divisible by $p^r$ with $r = \min (\nu_p (L), \nu_p
  (\boldsymbol{k}))$.
  
  Let $h (z) = (g (z)^{\ell} - g (z^p)^{\ell / p}) / \ell$ for $\ell \in
  p\mathbb{Z}$. Since $g (z) \in \mathbb{Z}_p [[z]]$, we have $g (z)^p
  \equiv g (z^p)$, which implies that $h$ has $p$-adically integer
  coefficients. Let $h_L \in \mathbb{Z}_p$ be the coefficient of $z^L$ in $h
  (z)$. Then, the coefficient of $z^L$ in $z \frac{\md}{\md z} h (z)$ is
  $L h_L$, which is divisible by $p^{\nu_p (L)}$.
  
  It therefore only remains to show that the coefficient of $z^L
  \boldsymbol{x}^{\boldsymbol{k}}$ in $F_1 (z, \boldsymbol{x})$ is divisible by
  $p^r$ with $r = \min (\nu_p (L), \nu_p (\boldsymbol{k}))$. Equivalently, if
  $C_L$ is the coefficient of $z^L$ in
  \begin{equation*}
    \frac{f_{\ell, \boldsymbol{k}} - f_{\ell / p, \boldsymbol{k}/ p}}{\ell} z
     \frac{\md}{\md z} [g (z)^{\ell}],
  \end{equation*}
  then we need to show that $\nu_p (C_L) \geq \min (\nu_p (L), \nu_p
  (\boldsymbol{k}))$. Observe that the $p$-adic valuation of the coefficient of
  $z^L$ in $z \frac{\md}{\md z} g (z)^{\ell}$ is at least $\max (\nu_p
  (\ell), \nu_p (L))$. Then, because $f$ satisfies the Gauss congruences for
  $p$,
  \begin{equation*}
    \nu_p (C_L) \geq \min (\nu_p (\ell), \nu_p (\boldsymbol{k})) - \nu_p
     (\ell) + \max (\nu_p (\ell), \nu_p (L)) \geq \min (\nu_p (L), \nu_p
     (\boldsymbol{k})),
  \end{equation*}
  which completes the proof.
\end{proof}

\begin{corollary}
  \label{cor:subst:rat}Let $g_j \in \mathbb{Q} (x)$ be nonzero. If the
  rational function $f \in \mathbb{Q} (\boldsymbol{x})$ has the Gauss property,
  then so does the rational function
  \begin{equation*}
    \left(\prod_{j = 1}^n \frac{x_j g'_j (x_j)}{g_j (x_j)} \right) f (g_1
     (x_1), \ldots, g_n (x_n)) .
  \end{equation*}
\end{corollary}

\begin{proof}
  Clearly, it suffices to show that, if $f \in \mathbb{Q} (z, \boldsymbol{x})$
  has the Gauss property, then, for any nonzero $g \in \mathbb{Q} (z)$, the
  rational function
  \begin{equation*}
    F (z, \boldsymbol{x}) = \frac{z g' (z)}{g (z)} f (g (z), \boldsymbol{x})
  \end{equation*}
  has the Gauss property.
  
  First, consider the case that $g (0) = 0$. Let $P, Q \in \mathbb{Z} [z,
  \boldsymbol{x}]$ and suppose that $f = P / Q$ has the Gauss property. Observe
  that there exists a vertex $\boldsymbol{v}$ of $N (Q)$ and a linear form
  $\alpha$ with $\alpha (1, 0, \ldots 0) > 0$, such that $\alpha
  (\boldsymbol{w}) > 0$ for all nonzero $\boldsymbol{w}$ in the proper cone $C$
  generated by $N (Q / (z, \boldsymbol{x})^{\boldsymbol{v}})$. It follows from
  Corollary~\ref{cor:NQ:cone} that the Laurent series expansion of $f$ with
  respect to $\boldsymbol{v}$ is supported on $C$. Since adding a constant does
  not affect the Gauss property, we may assume that this series has constant
  term $0$. Hence, the assumptions of Theorem~\ref{thm:subst} are satisfied
  for all but finitely many primes $p$. We conclude that $F (z, \boldsymbol{x})$
  has the Gauss property.
  
  Next, suppose that $g (z)$ has a pole at $z = 0$. By
  Proposition~\ref{prop:toroidal}, $F (z, \boldsymbol{x})$ has the Gauss
  property if and only if $F (1 / z, \boldsymbol{x})$ has the Gauss property.
  Let $h (z) = g (1 / z)$, and note that
  \begin{equation*}
    F (1 / z, \boldsymbol{x}) = - \frac{z h' (z)}{h (z)} f (h (z),
     \boldsymbol{x}) .
  \end{equation*}
  Since $h (0) = 0$, it follows from the previous case that $F (1 / z,
  \boldsymbol{x})$ and, hence, $F (z, \boldsymbol{x})$ has the Gauss property.
  
  It therefore remains to consider the case $g (0) = c \in
  \mathbb{Q}^{\times}$. In light of the first case, it suffices to consider
  $g (z) = z + c$. Observe that $g (z) = g_1 (g_2 (g_1 (z)))$ with $g_1 (z) =
  1 / z$ and $g_2 (z) = z / (1 + c z)$. Since the result holds for $g_1$ and
  $g_2$, we conclude that it also holds for $g$.
\end{proof}

\section{A proof of Minton's result}\label{sec:minton}

In this section, we reprove the following result of Minton
\cite{minton-cong}.

\begin{theorem}[Minton, 2014]
  \label{thm:minton}Let $f \in \mathbb{Q} (x)$. Then the following are
  equivalent:
  \begin{enumerate}
    \item \label{i:minton:gauss}$f$ has the Gauss property.
    
    \item \label{i:minton:p}The coefficients $a_n$ of the Laurent series
    expansion $f = \sum a_n x^n$ satisfy
    \begin{equation*}
      a_{n p} \equiv a_n \pmod{p}
    \end{equation*}
    for almost all primes $p$.
    
    \item \label{i:minton:log}$f$ is $f (0)$ plus a $\mathbb{Q}$-linear
    combination of functions of the form $x u' (x) / u (x)$, where $u \in
    \mathbb{Q} [x]$ is irreducible and $u (0) = 1$.
  \end{enumerate}
\end{theorem}

Note that statement \ref{i:minton:p} concerns only congruences modulo primes.
By the theorem, this already implies the Gauss congruences modulo prime
powers.

\begin{proof}
  That \ref{i:minton:log} implies \ref{i:minton:gauss} is a consequence of the
  case $m = 1$ of Theorem~\ref{thm:det:rat}. Since \ref{i:minton:gauss}
  obviously implies \ref{i:minton:p}, it remains to show that \ref{i:minton:p}
  implies \ref{i:minton:log}.
  
  Write $f = P / Q$ with $P, Q \in \mathbb{Z} [x]$. Assume that the
  congruences \ref{i:minton:p} hold or, equivalently, that
  \begin{equation}
    U_p (P / Q) \equiv P / Q \quad (\operatorname{mod} p), \label{eq:minton:U}
  \end{equation}
  where $U$ is the operator introduced in \eqref{eq:U}. It follows as in the
  proof of Proposition~\ref{prop:NPNQ} (which only relied on the Gauss
  congruences modulo primes, not prime powers) that $N (P) \subseteq N (Q)$.
  This implies that $P / Q$ has no pole in $x = 0$ and that $\deg (P)
  \leq \deg (Q)$. Since adding a constant to $f$ does not affect the
  result, we may assume that $Q (0) = 1$ and that $\deg (P) < \deg (Q)$. Then,
  $P / Q$ has a partial fraction expansion of the form
  \begin{equation*}
    \frac{P}{Q} = \sum_{i = 1}^r \sum_{j = 1}^{m_i} \frac{A_{i j}}{(1 -
     \alpha_i x)^j},
  \end{equation*}
  where the $\alpha_i \in \bar{\mathbb{Q}}$ are distinct algebraic numbers,
  $m_i \geq 1$ and $A_{i j} \in \bar{\mathbb{Q}}$. We first show that
  $m_i = 1$. To that end, note that, for all sufficiently large prime numbers
  $p$, we have that, for all $i$ and $j$, the norms of $\alpha_i$ and $A_{i
  j}$ are $p$-adic units, the $\alpha_i$ are distinct modulo $p$, and $p >
  m_i$. Let $p$ be a prime number satisfying these conditions. Since $p >
  m_i$, we then have, for all $j = 1, 2, \ldots, m_i$,
  \begin{eqnarray*}
    U_p \left(\frac{1}{(1 - \alpha_i x)^j} \right) & = & U_p \left(\sum_{k
    \geq 0} \binom{k + j - 1}{j - 1} \alpha_i^k x^k \right)\\
    & = & \sum_{k \geq 0} \binom{p k + j - 1}{j - 1} \alpha_i^{p k}
    x^k\\
    & \equiv & \sum_{k \geq 0} \alpha_i^{p k} x^k = \frac{1}{1 -
    \alpha_i^p x} \quad (\operatorname{mod} p) .
  \end{eqnarray*}
  As a consequence, we see that $U_p (P / Q)$, modulo $p$, is equal to a
  rational function with simple poles. From \eqref{eq:minton:U} we conclude
  that $P / Q$ has only simple poles as well.
  
  From now on, we may therefore write $A_i = A_{i 1}$ and have
  \begin{equation*}
    \frac{P}{Q} = \sum_{i = 1}^r \frac{A_i}{1 - \alpha_i x}, \quad U_p \left(\frac{P}{Q} \right) \equiv \sum_{i = 1}^r \frac{A_i}{1 - \alpha_i^p x} .
  \end{equation*}
  Moreover, the $k$th coefficient of $P / Q = \sum_{k \geq 0} f_k x^k$ is
  $f_k = \sum_{i = 1}^r A_i \alpha_i^k$. Since $f_k \in \mathbb{Q}$ is a
  $p$-adic integer, we have
  \begin{equation*}
    f_k \equiv f_k^p \equiv \sum_{i = 1}^r A_i^p \alpha_i^{p k} \quad
     (\operatorname{mod} p),
  \end{equation*}
  which implies that
  \begin{equation*}
    \frac{P}{Q} \equiv \sum_{i = 1}^r \frac{A_i^p}{1 - \alpha_i^p x} \quad
     (\operatorname{mod} p) .
  \end{equation*}
  Since $P / Q \equiv U_p (P / Q) \pmod{p}$, we conclude
  that $A_i^p \equiv A_i \pmod{p}$, for all $i = 1, 2,
  \ldots, r$. From Frobenius's density theorem, see, for instance,
  \cite[p.~134]{janusz-anf}, it follows that $A_i \in \mathbb{Q}$ for all
  $i$.
  
  Finally, let us group the $\alpha_i$ in Galois orbits under $\operatorname{Gal}
  (\bar{\mathbb{Q}} /\mathbb{Q})$. Suppose, say, that $\alpha_1, \ldots,
  \alpha_s$ is such a Galois orbit. Since $P / Q$ is Galois invariant, and the
  $A_i$ are rational, we must have $A_1 = A_2 = \ldots = A_s$. Hence, we
  conclude that $P / Q$ is a rational linear combination of functions of the
  form
  \begin{equation*}
    \sum_{i = 1}^s \frac{1}{1 - \alpha_i x} = s - x \frac{v' (x)}{v (x)},
  \end{equation*}
  where $v (x) = \prod_{i = 1}^s (1 - \alpha_i x)$. Moreover, $v \in
  \mathbb{Q} [x]$ because $\alpha_1, \ldots, \alpha_s$ form a Galois orbit.
\end{proof}

\begin{acknowledgements}
The first and third author would like to thank
the Max-Planck-Institute for Mathematics in Bonn, where this work was
initiated, for providing support and wonderful working conditions.
\end{acknowledgements}

\end{document}